\documentclass[a4paper,10pt]{amsart}
\usepackage[utf8]{inputenc}

\usepackage{graphicx,amssymb,amsmath,amsfonts,amsthm,amscd,hyperref,textcomp,relsize}
\usepackage{tikz-cd}
\usetikzlibrary{babel} 

\newtheorem{proposition}{Proposition}
\newtheorem{defn}{Definition}
\newtheorem{lemma}{Lemma}
\newtheorem{cor}{Corollary}
\newtheorem{thm}{Theorem}

\theoremstyle{remark}

\newcommand{\CC}{\mathbb C}

\newcommand{\Fq}{{\mathbb F}_q}

\newcommand{\Ql}{\bar{\mathbb Q}_\ell}
\newcommand{\Dbc}{D^b_c}

\newcommand{\FF}{\mathcal F}
\newcommand{\GGG}{\mathcal G}
\newcommand{\HH}{\mathcal H}
\newcommand{\LL}{\mathcal L}
\newcommand{\R}{\mathrm R}

\newcommand{\AAA}{\mathbb A}
\newcommand{\GG}{\mathbb G}
\newcommand{\Gm}{{\mathbb G}_m}
\newcommand{\Gmk}{{\mathbb G}_{m,k}}
\newcommand{\ZZ}{\mathbb Z}
\newcommand{\Nm}{\mathrm N}

\newcommand{\PP}{\mathbb P}
\newcommand{\HHH}{\mathrm H}

\title{On a generalization of Jacobi sums}
\author{Antonio Rojas-Le\'on}
\address{Dpto. de \'Algebra, Fac. de Matem\'aticas \\
  Universidad de Sevilla \\
  c/Tarfia, s/n \\
  41012 Sevilla, Spain \\
  \tt{arojas@us.es}}

\begin{document}

\maketitle

\renewcommand{\thefootnote}{}
\footnote{Mathematics Subject Classification: 11L05, 11L07, 11T24}
\footnote{Partially supported by MTM2016-75027-P (Ministerio de Econom\'{\i}a y Competitividad and FEDER) and US-1262169 (Consejería de Econom\'{\i}a, Conocimiento, Empresas y Universidad de la Junta de Andalucía and FEDER)}

\begin{abstract}
We prove an estimate for multi-variable multiplicative character sums over affine subspaces of $\mathbb A^n_k$, which generalize the well known estimates for both classical Jacobi sums and one-variable polynomial multiplicative character sums.
\end{abstract}

\section{Introduction}

Let $k=\Fq$ be a finite field, with $q=p^a$ a prime power. Given $n$ non-trivial multiplicative characters $\chi_1,\ldots,\chi_n:\Fq^\times\to\CC^\times$, the classical Jacobi sum is given by \cite[10.1]{berndt1998gauss}
$$
J(\chi_1,\ldots,\chi_n):=\sum_{\stackrel{x_1,\ldots,x_n\in k}{x_1+\cdots+x_n=1}}\chi_1(x_1)\cdots\chi_n(x_n).
$$

More generally, given a hyperplane $L\subseteq\AAA^n_k$ defined by the equation $a_1x_1+\cdots+a_nx_n=b$ with $a_1,\ldots,a_n,b\neq 0$, the sum
\begin{equation}\label{sum}
S(L;\chi_1,\ldots,\chi_n):=\sum_{\mathbf x\in L}\chi_1(x_1)\cdots\chi_n(x_n)
\end{equation}
can be rewritten, via the change of variables $y_i=a_ix_i/b$, as
$$
\chi_1\cdots\chi_n(b)\overline\chi_1(a_1)\cdots\overline\chi_n(a_n)J(\chi_1,\ldots,\chi_n).
$$

It is well known \cite[Theorem 10.3.1]{berndt1998gauss} that the Jacobi sum can be expressed in terms of Gauss sums:
$$
J(\chi_1,\ldots,\chi_n)=\left\{\begin{array}{ll}
                                G(\chi_1)\cdots G(\chi_n)/G(\chi_1\cdots\chi_n) & \text{if }\chi_1\cdots\chi_n\neq\mathbf 1 \\
                                -G(\chi_1)\cdots G(\chi_n)/q & \text{if }\chi_1\cdots\chi_n=\mathbf 1
                               \end{array}\right.
$$
where $G(\chi):=\sum_{x\in k}\chi(x)\psi(x)$ is the Gauss sum associated to $\chi$, $\psi:k\to\CC^\times$ being an additive character. In particular, since $|G(\chi)|=\sqrt{q}$, we get
$$
|S(L;\chi_1,\ldots,\chi_n)|=\left\{\begin{array}{ll}
                                q^{(n-1)/2} & \text{if }\chi_1\cdots\chi_n\neq\mathbf 1 \\
                                q^{(n-2)/2} & \text{if }\chi_1\cdots\chi_n=\mathbf 1
                               \end{array}\right.
$$

Inspired by a question from Ivan Meir, in this article we study, more generally, sums of the form (\ref{sum}) where $L\subseteq\AAA^n_k$ is a sufficiently general affine subspace of arbitrary dimension. Before stating the main results, let us give a suitable definition of ``sufficiently general''. For any $i=1,\ldots,n$ let $H_i\subseteq\AAA^n_k$ be the hyperplane defined by the equation $x_i=0$ and, for every subset $I\subseteq\{1,\ldots,n\}$, denote by $H_I$ the intersection $\cap_{i\in I}H_i$. By convention, we set $\dim(\emptyset)=-1$.

\begin{defn}
Let $L\subseteq\AAA^n_k$ be an affine subspace of dimension $d$.
\begin{enumerate}
 \item We say that $L$ is in {\bf general position} (with respect to the coordinate hyperplanes) if, for every $I\subseteq\{1,\ldots,n\}$ with $|I|\leq d+1$, we have $\dim(L\cap H_I)=d-|I|$.
 \item We say that $L$ is in {\bf general position among its translates} if, for every $I\subseteq\{1,\ldots,n\}$ with $|I|\leq d+1$, we have $\dim(L\cap H_I)\leq d-|I|$ (so either $\dim(L\cap H_I)= d-|I|$ or $L\cap H_I=\emptyset$).
 \end{enumerate}
\end{defn}

Here is some justification for the chosen terminology. Consider the (affine) Grassmanian variety $G(n,d)$ parameterizing all $d$-dimensional affine subspaces of $\AAA^n_k$. The functions $f_I:G(n,d)\to\mathbb Z$ given by $f_I(L)=\dim(L\cap H_I)$ define a stratification of $G(n,d)$ by locally closed subsets, such that all functions $f_I$ are constant in each stratum. Then the subspaces in general position are those that belong to the dense open stratum. Similarly, $L$ is in general position among its translates if it is in the dense open stratum of the restriction of the stratification to the closed subset of $G(n,d)$ consisting of all translates of $L$.

If $A\cdot\mathbf x=\mathbf b$ is a system of linear equations defining $L$, where $A$ is an $m\times n$ matrix of maximal rank (with $m=n-d$), then $L$ is in general position if and only if all $m\times m$ minors of the augmented matrix $(A|\mathbf b)$ are non-zero, and $L$ is in general position among its translates if and only if $\mathbf b$ is not in any proper subspace of $k^m$ generated by a subset of the columns of $A$. In particular, for a fixed $A$, there is a dense open set of $\mathbf b\in k^m$ such that the affine subspace defined by $A\cdot\mathbf x=\mathbf b$ is in general position among its translates.

For example, if $d=0$ (i.e. $L=\{P\}$ is a single point) then $L$ is in general position if and only if it is in general position among its translates, if and only if all coordinates of $P$ are non-zero. If $d=1$ (i.e. $L$ is a line) then $L$ is in general position among its translates if and only if all points of $L$ have at most one coordinate equal to $0$, and it is in general position if, additionally, it is not parallel to any of the coordinate hyperplanes. If $L$ is a hyperplane with equation $a_1x_1+\cdots+a_nx_n=b$ then $L$ is in general position among its translates if and only if $b\neq 0$, and it is in general position if, additionally, all $a_i\neq 0$.

For any $L\subseteq\AAA^n_k$, we define the sum
$$
S(L;\chi_1,\ldots,\chi_n):=\sum_{\mathbf x\in L(k)}\chi_1(x_1)\cdots\chi_n(x_n)
$$
and, more generally, for every finite extension $k_r$ of $k$ of degree $r$,
$$
S_r(L;\chi_1,\ldots,\chi_n):=\sum_{\mathbf x\in L(k_r)}\chi_1(\Nm_{k_r/k}(x_1))\cdots\chi_n(\Nm_{k_r/k}(x_n))
$$
and we construct the associated $L$-function
$$
L(L;\chi_1,\ldots,\chi_n):=\exp\sum_{r=1}^\infty \frac{T^r}{r}S_r(L;\chi_1,\ldots,\chi_n).
$$

The main result of this article is the following
\begin{thm}\label{thm1}
Suppose that $L$ is in general position among its translates. Then $L(L;\chi_1,\ldots,\chi_n)^{(-1)^d}$ is a polynomial, whose degree is given by
$$
D_L:=(-1)^d+\sum_{j=1}^d (-1)^{d+j}a_j
$$
where $a_j$ is the number of subsets $I\subseteq\{1,\ldots,n\}$ with $|I|=j$ such that $L\cap H_I\neq\emptyset$, and all whose reciprocal roots are algebraic integers of absolute value $q^{i/2}$ for some $i\leq d$.
\end{thm}

If $L$ is in general position, we can give a more precise result:

\begin{thm}\label{thm2}
Suppose that $L$ is in general position. Then $L(L;\chi_1,\ldots,\chi_n)^{(-1)^d}$ is a polynomial of degree ${n-1}\choose{d}$, all whose reciprocal roots are algebraic integers. If $\chi_1\cdots\chi_n$ is non-trivial, then all of its reciprocal roots have absolute value $q^{d/2}$. If $\chi_1\cdots\chi_n$ is trivial, then ${n-2}\choose d$ of its reciprocal roots have absolute value $q^{d/2}$ and ${n-2}\choose{d-1}$ of them have absolute value $q^{(d-1)/2}$.
\end{thm}

This matches the known cases where $d=n-1$ (classical Jacobi sums) or $d=1$ (one-variable polynomial sums, cf. \cite[\S 1]{katz2002estimates} for the case where all $\chi_i$ are equal). Since the sums $S(L;\chi_1,\ldots,\chi_n)$ can be written (up to sign) as the sum of the reciprocal roots of $L(L;\chi_1,\ldots,\chi_n)^{(-1)^d}$, as the main consequence of the theorems we get an estimate for them:

\begin{cor}
Suppose that $L$ is in general position among its translates. Then we have an estimate
$$
|S(L;\chi_1,\ldots,\chi_n)|\leq D_L\cdot q^{d/2}.
$$
If $L$ is in general position, we have
$$
|S(L;\chi_1,\ldots,\chi_n)|\leq {{n-1}\choose{d}}\cdot q^{d/2}
$$
if $\chi_1\cdots\chi_n$ is non-trivial, and
$$
|S(L;\chi_1,\ldots,\chi_n)|\leq {{n-2}\choose{d}}\cdot q^{d/2}+{{n-2}\choose{d-1}}\cdot q^{(d-1)/2}
$$
if $\chi_1\cdots\chi_n$ is trivial.
\end{cor}
We also have the following variant, where $L$ is given by a parametrization:
\begin{cor}
 Let $L_1,\ldots,L_n:\AAA^d_k\to\AAA^1_k$ be affine linear forms, with $L_i(\mathbf t)=a_{i,1}t_1+\cdots+a_{i,d}t_d+b_i$, and let $V_i\subseteq \AAA^d_k$ be the hyperplane defined by $L_i(\mathbf t)=0$. Suppose that the affine map $\AAA^d_k\to\AAA^n_k$ defined by the $L_i$ is injective (that is, that the matrix $(a_{ij})$ has rank $d$), and that for every $I\subseteq\{1,\ldots,n\}$ with $|I|\leq d+1$ we have $\dim\left(\bigcap_{i\in I}V_i\right)\leq d-|I|$. Then we have an estimate
 $$
 \left|\sum_{\mathbf t\in k^d}\chi_1(L_1(\mathbf t))\cdots\chi_n(L_n(\mathbf t))\right|\leq D_L\cdot q^{d/2}
 $$
 where $D_L:=(-1)^d+\sum_{j=1}^d (-1)^{d+j}a_j$ and $a_j$ is the number of subsets $I\subseteq\{1,\ldots,n\}$ with $|I|=j$ such that $\bigcap_{i\in I}V_i\neq\emptyset$.
\end{cor}

Note that, even in the case where all $\chi_i$ are equal to the same character $\chi$, the subvariety of $\AAA^d_k$ defined by $L_1(\mathbf t)\cdots L_n(\mathbf t)=0$ is highly singular, so the estimates in \cite{katz2002estimates} and \cite{rojas2005estimates} for multi-variable multiplicative character sums do not give good bounds for these sums.

\section{Proof of Theorem \ref{thm1}}

In this section we will prove Theorem \ref{thm1}. We keep the same notation as in the previous section, in particular, we fix a finite field $k=\Fq$ of characteristic $p$, and $n$ non-trivial multiplicative characters $\chi_1,\ldots,\chi_n:\Fq^\times\to\CC^\times$. Let $L\subseteq\AAA^n_k$ be an affine subspace of dimension $d$, defined as the solution set of the linear system
$$
A\cdot\mathbf{x}=\mathbf {b}
$$
where $A$ is an $m\times n$ matrix of maximal rank $m:=n-d$. 

Fix a prime $\ell\neq p$. We will work on the category of $\ell$-adic constructible sheaves on $\AAA^m_k$ and, more generally, its derived category $\Dbc(\AAA^m_k,\Ql)$. For every $i=1,\ldots,n$ there is a rank 1, smooth sheaf $\LL_{\chi_i}$ on $\Gmk$ (the Kummer sheaf) whose Frobenius trace at a point $t\in\Gm(k_r)=k_r^\times$ is $\chi_i(\Nm_{k_r/k}(t))$ \cite{deligne569application}. We will also denote by $\LL_{\chi_i}$ its extension by zero to $\AAA^1_k$. Consider the product $\LL_{\chi_1}\boxtimes\cdots\boxtimes\LL_{\chi_n}$ on $\AAA^n_k$, then the sum of the traces of the Frobenius action on its stalks at the points of $L(k_r)$ is precisely $S_r(L;\chi_1,\ldots,\chi_n)$, so $L(L;\chi_1,\ldots,\chi_n)$ is the $L$-function of the sheaf $\LL_{\chi_1}\boxtimes\cdots\boxtimes\LL_{\chi_n}$ on $L$. By the Grothendieck-Lefschetz trace formula we get
$$
L(L;\chi_1,\ldots,\chi_n)=\prod_{i=0}^{2d}\det(1-T\cdot F|\HHH^i_c(L\otimes\bar k,\LL_{\chi_1}\boxtimes\cdots\boxtimes\LL_{\chi_n}))^{(-1)^{i+1}}
$$
where $L\otimes\bar k\subseteq\AAA^n_{\bar k}$ is obtained from $L$ by extension of scalars, so Theorem \ref{thm1} is then a consequence of the following result, since $\LL_{\chi_1}\boxtimes\cdots\boxtimes\LL_{\chi_n}$ is pure of weight $0$ (so $\HHH^i_c(L,\LL_{\chi_1}\boxtimes\cdots\boxtimes\LL_{\chi_n})$ is mixed of weights $\leq i$ by \cite{deligne1980conjecture}).

\begin{thm}\label{thm3}
Suppose that $L$ is in general position among its translates. Then $\HHH^i_c(L\otimes\bar k,\LL_{\chi_1}\boxtimes\cdots\boxtimes\LL_{\chi_n})=0$ for $i\neq d$, and $\HHH^d_c(L\otimes\bar k,\LL_{\chi_1}\boxtimes\cdots\boxtimes\LL_{\chi_n})$ has dimension $D_L$.
\end{thm}

Let $K=\R A_!(\LL_{\chi_1}\boxtimes\cdots\boxtimes\LL_{\chi_n})\in\Dbc(\AAA^m_k,\Ql)$, where $A$ is viewed as a linear map $\AAA^n_k\to\AAA^m_k$. Then $\HHH^i_c(L\otimes\bar k,\LL_{\chi_1}\boxtimes\cdots\boxtimes\LL_{\chi_n})$ is the stalk of $\HH^i(K)$ at $\mathbf b$. Recall that an object $P$ in $\Dbc(\AAA^m_k,\Ql)$ is called \emph{perverse} if the dimension of the support of the cohomology sheaves $\HH^i(P)$ and $\HH^i(DP)$ is $\leq -i$ for every $i\in{\mathbb Z}$, where $DP$ is the Verdier dual of $P$ \cite{beilinson1982faisceaux}.

\begin{proposition} The shifted object $K[n]\in\Dbc(\AAA^m_k,\Ql)$ is perverse.
 \end{proposition}

 \begin{proof}
 Let $\mathbf a_1,\ldots,\mathbf a_n\in k^m$ be the columns of $A$, and let $\LL_{\overline\chi_i(\mathbf a_i)}$ be the pull-back of the Kummer sheaf $\LL_{\overline\chi_i}$ via the linear form $\AAA^m_k\to\AAA^1_k$ given by ${\mathbf t}\mapsto \mathbf a_i\cdot{\mathbf t}$ (which we will also denote by $\mathbf a_i$). We define
 $$P:=\LL_{\overline\chi_1(\mathbf a_1)}\otimes\cdots\otimes\LL_{\overline\chi_n(\mathbf a_n)}[m]\in\Dbc(\AAA^m_k,\Ql).$$
  This is a perverse object: we have $P=j_!j^\ast P$, where $j:U\hookrightarrow\AAA^m_k$ is the inclusion of the dense open complement of the hyperplanes $\{A_i\cdot{\mathbf t}=0;\; i=1,\ldots,n\}$. Since $\LL_{\overline\chi_1(A_1)}\otimes\cdots\otimes\LL_{\overline\chi_n(A_n)}$ is smooth on $U$, $j^* P$ is perverse on $U$ by \cite[4.0]{beilinson1982faisceaux}, and so is $P$ by \cite[Corollaire 4.1.3]{beilinson1982faisceaux}, since $j$ is quasi-finite and affine.
  
  We will show that $K[n]$ is geometrically isomorphic to the Fourier transform \cite{laumon1987transformation} of $P$. Since the Fourier transform preserves perversity \cite[1.3.2.3]{laumon1987transformation}, this will prove the proposition. Equivalently, since the Fourier transform is (geometrically) involutive, we will show that the Fourier transform of $K[n]$ is geometrically isomorphic to $P$.
 
 Fix an additive character $\psi:k\to\CC$, with respect to which we will take the Fourier Transform of $K[n]$. This is defined as
 $$
 FT_\psi(K[n])=\R\pi_{2!}(\mu^\ast\LL_\psi\otimes\pi_1^\ast K[n])[m]
 $$
 where $\pi_1,\pi_2:\AAA^m_k\times\AAA^m_k\to\AAA^m_k$ are the projections, $\mu:\AAA^m_k\times\AAA^m_k\to\AAA^1_k$ is the multiplication map $\mu(\mathbf t,\mathbf u)=\mathbf t\cdot\mathbf u=\sum_{i=1}^m t_iu_i$ and $\LL_\psi$ is the Artin-Schreier smooth sheaf on $\AAA^1_k$ associated to $\psi$ \cite{deligne569application}. We have
 $$
 FT_\psi(K[n])[-n-m]=\R\pi_{2!}(\mu^\ast\LL_\psi\otimes\pi_1^\ast\R A_!(\LL_{\chi_1}\boxtimes\cdots\boxtimes\LL_{\chi_n}))\cong
 $$
 $$
 \cong\R\pi_{2!}(\mu^\ast\LL_\psi\otimes \R(A,Id)_!\tilde\pi_1^\ast(\LL_{\chi_1}\boxtimes\cdots\boxtimes\LL_{\chi_n}))\cong
 $$
 $$
 \cong\R\pi_{2!}\R(A,Id)_!(\tilde\mu^\ast\LL_\psi\otimes\tilde\pi_1^\ast(\LL_{\chi_1}\boxtimes\cdots\boxtimes\LL_{\chi_n}))\cong
 $$
 $$
 \cong\R\tilde\pi_{2!}(\tilde\mu^\ast\LL_\psi\otimes\tilde\pi_1^\ast(\LL_{\chi_1}\boxtimes\cdots\boxtimes\LL_{\chi_n}))
 $$
 by proper base change and the projection formula, where $\tilde\pi_1:\AAA^n_k\times\AAA^m_k\to\AAA^n_k$ and $\tilde\pi_2:\AAA^n_k\times\AAA^m_k\to\AAA^m_k$ are the projections and $\tilde\mu:\AAA^n_k\times\AAA^m_k\to\AAA^1_k$ is given by $(\mathbf x,\mathbf u)\mapsto (A\mathbf x)\cdot\mathbf u=\sum_{i=1}^n x_i(\mathbf a_i\cdot \mathbf u)$. In particular, $\tilde\mu^\ast\LL_\psi=\bigotimes_{i=1}^n\tilde\mu_i^\ast\LL_\psi$, where $\tilde\mu_i=\sigma_i(\varpi_i,Id):\AAA^n_k\times\AAA^m_k\to\AAA^1_k\times\AAA^m_k\to\AAA^1_k$, $\varpi_i:\AAA^n_k\to\AAA^1_k$ being the $i$-th projection and $\sigma_i:\AAA^1_k\times\AAA^m_k\to\AAA^1_k$ being given by $(x,\mathbf u)\mapsto x(\mathbf a_i\cdot \mathbf u)$. We then have
 $$
 FT_\psi(K[n])[-n-m]\cong\R\tilde\pi_{2!}(\tilde\mu_1^\ast\LL_\psi\otimes\cdots\otimes\tilde\mu_n^\ast\LL_\psi\otimes\tilde\pi_1^\ast(\LL_{\chi_1}\boxtimes\cdots\boxtimes\LL_{\chi_n}))\cong
 $$
 $$
 \cong\R\tilde\pi_{2!}((\varpi_1,Id)^\ast\sigma_i^\ast\LL_\psi\otimes\cdots\otimes(\varpi_n,Id)^\ast\sigma_n^\ast\LL_\psi\otimes(\varpi_1,Id)^\ast\hat\pi_1^\ast\LL_{\chi_1}\otimes\cdots\otimes(\varpi_n,Id)^\ast\hat  \pi_1^\ast\LL_{\chi_n})
 $$
 where $\hat\pi_1:\AAA^1_k\times\AAA^m_k\to\AAA^1_k$, $\hat\pi_2:\AAA^1_k\times\AAA^m_k\to\AAA^m_k$ are the projections. By the K\"unneth formula, this is
 $$
 \R\hat\pi_{2!}(\sigma_1^\ast\LL_\psi\otimes\hat\pi_1^\ast\LL_{\chi_1})\otimes\cdots\otimes\R\hat\pi_{2!}(\sigma_n^\ast\LL_\psi\otimes\hat\pi_1^\ast\LL_{\chi_n})
 $$
 so it only remains to show that
 \begin{equation}\label{remainstoshow}
 \R\hat\pi_{2!}(\sigma_i^\ast\LL_\psi\otimes\hat\pi_1^\ast\LL_{\chi_i})\cong\LL_{\overline\chi_i(\mathbf a_i)}[-1]=\mathbf a_i^\ast\LL_{\overline\chi_i}[-1]
 \end{equation}
 geometrically for every $i=1,\ldots,n$. If $\hat\mu:\AAA^1_k\times\AAA^1_k\to\AAA^1_k$ is the multiplication map, we can write
 $$
 \R\hat\pi_{2!}(\sigma_i^\ast\LL_\psi\otimes\hat\pi_1^\ast\LL_{\chi_i})=\R\hat\pi_{2!}((Id,\mathbf a_i)^\ast\hat\mu^\ast\LL_\psi\otimes\hat\pi_1^\ast\LL_{\chi_i})\cong
 $$
 $$
 \cong\R\hat\pi_{2!}((Id,\mathbf a_i)^\ast(\hat\mu^\ast\LL_\psi\otimes\pi_1^\ast\LL_{\chi_i}))\cong \mathbf a_i^\ast\R\pi_{2!}(\hat\mu^\ast\LL_\psi\otimes\pi_1^\ast\LL_{\chi_i})=\mathbf a_i^\ast FT_\psi(\LL_{\chi_i})[-1]
 $$
 again by proper base change, where $\pi_1,\pi_2$ are now the projections $\AAA^1_k\times\AAA^1_k\to\AAA^1_k$. Then (\ref{remainstoshow}) follows from the fact \cite[1.4.3.1]{laumon1987transformation} that the Fourier transform of $\LL_\chi$ is geometrically isomorphic to $\LL_{\overline\chi}$ for any non-trivial multiplicative character $\chi$.
 \end{proof}

 Next, we will show that, if $L$ (defined by $A\cdot\mathbf{x}=\mathbf {b}$) is in general position among its translates, then the cohomology sheaves of $K$ are smooth at $\mathbf b\in\AAA^m_k$.
 
\begin{proposition} \label{smooth}
 Let $\mathbf b\in k^m$ such that the affine subspace $L\subseteq\AAA^n_k$ defined by $A\cdot\mathbf{x}=\mathbf {b}$ is in general position among its translates. Then the cohomology sheaf $\HH^i(K)$ is smooth at $\mathbf b$ for every $i\in\ZZ$.
\end{proposition}

\begin{proof}
 Let $\overline X\subseteq \PP^n_k\times\AAA^m_k$ be the projective closure of the graph $X=\Gamma_A$ of the map $A:\AAA^n_k\to\AAA^m_k$. Then $K=\R\pi_{2\ast}\FF$, where $\FF$ is the extension by zero to $\overline X$ of the sheaf $\LL_{\chi_1(x_1)}\otimes\cdots\otimes\LL_{\chi_n(x_n)}$ on $X$. We claim that $\pi_2:\overline X\to\AAA^m_k$ is locally acyclic for $\FF$ at every point of its fibre at $\mathbf b$ \cite[2.12]{deligne569finitude}. By \cite[A.2.2]{deligne569finitude}, this implies that all cohomology sheaves of $K$ are smooth at $\mathbf b$. 
 
 Following \cite[3.7.3]{deligne1980conjecture} we will show that, over some neighborhood $U$ of $\mathbf b$ in $\AAA^m_k$, locally for the \'etale topology on $\overline X$, the pair $(\overline X,\FF)$ is $U$-isomorphic to the product of $U$ and a scheme endowed with a $\Ql$-sheaf $\GGG$ (in other words, that $\FF$ is \'etale-locally isomorphic to a sheaf independent of $\mathbf b$). We consider different cases depending on the position of the point ${\mathbf x}\in\pi_2^{-1}(\mathbf b)$ (viewed in $\PP^n_k$).
 
  a) Suppose that $\mathbf x\in\AAA^n_k$ has all its coordinates different from zero. Then $\FF$ is smooth in a neighborhood of $\mathbf x$, and in fact can be trivialized by taking a finite \'etale map, so the assertion is clear in this case.
  
  b) Suppose that $\mathbf x\in\AAA^n_k$ has some coordinates equal to zero. Without loss of generality, assume that $x_1=\cdots=x_e=0$ and $x_i\neq 0$ for $e+1 \leq i \leq n$. Then $\LL_{\chi_{e+1}(x_{e+1})}\otimes\cdots\otimes\LL_{\chi_n(x_n)}$ is smooth at $\mathbf x$, so $\FF$ is \'etale-localy isomorphic to $\LL_{\chi_1(x_1)}\otimes\cdots\otimes\LL_{\chi_e(x_e)}$.
  
  Since $L$ is in general position among its translates, $L\cap H_{\{1,\ldots,e\}}$ has dimension $d-e$. In particular, $t_1=x_1,\ldots,t_e=x_e$ is part of a system of parameters of $L$ at $\mathbf x$. Solving the system $A\cdot\mathbf x=\mathbf z$ for a generic $\mathbf z$ we get a parametrization
  $$\phi_{\mathbf z}:(t_1,\ldots,t_d)\mapsto(t_1,\ldots,t_e,\alpha_{e+1}(\mathbf z)+\sum_{i=1}^d c_{e+1,i}t_i,\ldots,\alpha_n(\mathbf z)+\sum_{i=1}^d c_{n,i}t_i)
 $$
 of $\pi_2^{-1}(\mathbf z)$ for $\mathbf z$ in a neighborhood $U$ of $\mathbf b$. The pull-back of $\LL_{\chi_1(x_1)}\otimes\cdots\otimes\LL_{\chi_e(x_e)}$ to $\AAA^d_k\times U$ via the map $(\mathbf t,\mathbf z)\mapsto(\phi_{\mathbf z}(\mathbf t),\mathbf t)$ is the $\GGG$ we are looking for.

 c) Suppose that $\mathbf x=(0:x_1:\ldots:x_n)\in\PP^n_k\backslash\AAA^n_k$. Without loss of generality, assume that $x_1=\ldots=x_e=0$ and $x_i\neq 0$ for $e+1\leq i\leq n$ for some $0\leq e\leq n-1$. Dehomogenizing with respect to $x_n$ and letting $y_i=x_i/x_n$ for $0\leq i\leq n-1$ be the new affine variables, we see that 
 $$
 \LL_{\chi_1(x_1/x_0)}\otimes\cdots\otimes\LL_{\chi_n(x_n/x_0)}\cong \LL_{\chi_1(y_1)}\otimes\cdots\otimes\LL_{\chi_{n-1}(y_{n-1})}\otimes\LL_{\overline{\chi_1\cdots\chi_n}(y_0)}
 $$
 which is \'etale-locally isomorphic to $\LL_{\chi_1(y_1)}\otimes\cdots\otimes\LL_{\chi_{e}(y_{e})}\otimes\LL_{\overline{\chi_1\cdots\chi_n}(y_0)}$.
 
 Note that the intersection of ${\overline L}$ (the projective closure of $L$) with the hyperplane at infinity $H_\infty=\{y_0=0\}$, of dimension $d-1$, does not depend on $\mathbf b$, only on the map $A$. With respect to the new coordinates, $\overline L\backslash\{x_0=0\}$ is the set of solutions of the system of equations
 $$
 \left(-\mathbf b | A'\right){\mathbf y}=-\mathbf a_n
 $$
 where $A'$ is the matrix $A$ with its last column $\mathbf a_n$ removed.
 
 Suppose that $\dim (\overline L\cap H_\infty\cap \overline H_{\{1,\ldots,e\}})=d-1-e$. Then $\{y_0,y_1,\ldots,y_e\}$ is part of a system of parameters of $\overline L$ at $\mathbf x$, and we conclude as in (b) (note that $\chi_1\cdots\chi_n$ might be trivial, but (b) is proved without any non-triviality condition on the $\chi_i$'s).
 
 Finally, suppose that $\dim (\overline L\cap H_\infty\cap \overline H_{\{1,\ldots,e\}})>d-1-e$.  Solving the system $\left(-\mathbf z | A'\right){\mathbf y}=-\mathbf a_n$ via Gaussian elimination, we get a parametrization $\phi_{\mathbf z}(\mathbf t)=(\phi_{\mathbf z,0}(\mathbf t),\phi_1(\mathbf t),\ldots,\phi_{n-1}(\mathbf t))$ of $\pi_{2}^{-1}(\mathbf z)$ for $\mathbf z$ in a neighborhood of $\mathbf b$, where $\phi_1,\ldots,\phi_{n-1}$ do not depend on $\mathbf z$. Furthermore, for $\mathbf z\neq(0,\ldots,0)$ (which is the case in a neighborhood of $\mathbf b$, since an affine subspace which is in general position among its translates can never contain the origin), $\phi_{z,0}(\mathbf t)$ has the form $\phi_0(\mathbf t)/\alpha(\mathbf z)$ where $\phi_0(\mathbf t)$ does not depend on $\mathbf z$ and $\alpha(\mathbf z)=\alpha_1 z_1+\cdots+\alpha_m z_m$ is linear on $\mathbf z$ and $\alpha(\mathbf b)\neq 0$.
 
 The pull-back of $\LL_{\chi_1(y_1)}\otimes\cdots\otimes\LL_{\chi_{e}(y_{e})}\otimes\LL_{\overline{\chi_1\cdots\chi_n}(y_0)}$ to $\AAA^d_k\times\AAA^r_k$ via $\phi(\mathbf t,\mathbf z)=(\phi_{\mathbf z}(\mathbf t),\mathbf z)$ is
 $$
 \LL_{\chi_1(\phi_1(\mathbf t))}\otimes\cdots\otimes\LL_{\chi_{e}(\phi_e(\mathbf t))}\otimes\LL_{\overline{\chi_1\cdots\chi_n}(\phi_0(\mathbf t))}\otimes\LL_{{\chi_1\cdots\chi_n}(\alpha(\mathbf z))}
 $$
 where the last factor is smooth in a neighborhood $U$ of $\mathbf b$. It is then \'etale-locally isomorphic to $\LL_{\chi_1(\phi_1(\mathbf t))}\otimes\cdots\otimes\LL_{\chi_{e}(\phi_e(\mathbf t))}\otimes\LL_{\overline{\chi_1\cdots\chi_n}(\phi_0(\mathbf t))}$ over $U$, which is independent of $\mathbf z$.
  \end{proof}

  Since $K[n]$ is perverse (so all its cohomology sheaves except for $\HH^{-m}(K)$ are supported on proper closed subsets), proposition \ref{smooth} implies than $\HH^{i}(K)_{\mathbf b}=0$ for $i\neq n-m=d$, which proves the first part of Theorem \ref{thm3}. The dimension of $\HHH^d_c(L\otimes\bar k,\LL_{\chi_1}\boxtimes\cdots\boxtimes\LL_{\chi_n})$ is then $(-1)^d$ times its Euler characteristic $\chi_c(L\otimes\bar k,\LL_{\chi_1}\boxtimes\cdots\boxtimes\LL_{\chi_n})$.

Let $H:=\bigcup_{i=1}^n H_i\subseteq\AAA^n_k$ be the union of the coordinate hyperplanes. We will show that, for any affine subvariety $L\subseteq\AAA^n_k$, the Euler characteristic $\chi_c(L\otimes\bar k,\LL_{\chi_1}\boxtimes\cdots\boxtimes\LL_{\chi_n})$ is just the Euler characteristic of $U:=L\backslash H$. We proceed by induction on $n$: for $n=1$ or, more generally, whenever $d=0$ or $d=n$, the statement is clear (in the latter case, $\chi_c(\AAA^n_{\bar k},\LL_{\chi_1}\boxtimes\cdots\boxtimes\LL_{\chi_n})\cong\chi_c(\AAA^1_{\bar k},\LL_{\chi_1})\cdots\chi_c(\AAA^1_{\bar k},\LL_{\chi_n})=0=\chi_c(\GG_{m,\bar k}^n)=\chi_c(L\backslash H)$).

In the general case, assume $d<n$ and pick some $i\in\{1,\ldots,n\}$ such that the projection $\varpi$ of $L$ on the hyperplane $H_i$ is injective. Without loss of generality, assume $i=n$. Let $\pi=\pi_n:\AAA^n_k\to\AAA^1_k$ be the projection, then
$$
\R\pi_{L!}(\LL_{\chi_1}\boxtimes\cdots\boxtimes\LL_{\chi_n})\cong\LL_{\chi_n}\otimes\R\pi_{L!}(\LL_{\chi_1}\boxtimes\cdots\boxtimes\LL_{\chi_{n-1}}\boxtimes\Ql).
$$
by the projection formula. Since $\LL_{\chi_n}$ is smooth of rank $1$ on $\GG_m$ and tamely ramified at both $0$ and $\infty$, the Grothendieck-Ogg-Shafarevic formula implies that $\chi_c(\AAA^1_{\bar k},K\otimes\LL_{\chi_n})=\chi_c(\GG_{m,\bar k},K)$ for any object $K\in\Dbc(\AAA^1_{\bar k},\Ql)$. In particular,
$$
\chi_c(L\otimes\bar k,\LL_{\chi_1}\boxtimes\cdots\boxtimes\LL_{\chi_n})=\chi_c(\AAA^1_{\bar k},\R\pi_{L!}(\LL_{\chi_1}\boxtimes\cdots\boxtimes\LL_{\chi_n}))=
$$
$$
=\chi_c(\GG_{m,\bar k},\R\pi_{L!}(\LL_{\chi_1}\boxtimes\cdots\boxtimes\LL_{\chi_{n-1}}\boxtimes\Ql))=\chi_c((L\backslash H_n)\otimes\bar k,\LL_{\chi_1}\boxtimes\cdots\boxtimes\LL_{\chi_{n-1}}\boxtimes\Ql)=
$$
$$
=\chi_c(L\otimes\bar k,\LL_{\chi_1}\boxtimes\cdots\boxtimes\LL_{\chi_{n-1}}\boxtimes\Ql)-\chi_c((L\cap H_n)\otimes\bar k,\LL_{\chi_1}\boxtimes\cdots\boxtimes\LL_{\chi_{n-1}})=
$$
$$
=\chi_c(\varpi(L)\otimes\bar k,\LL_{\chi_1}\boxtimes\cdots\boxtimes\LL_{\chi_{n-1}})-\chi_c((L\cap H_n)\otimes\bar k,\LL_{\chi_1}\boxtimes\cdots\boxtimes\LL_{\chi_{n-1}})=
$$
$$
=\chi_c(\varpi(L)\backslash\bigcup_{i=1}^{n-1}H_i)-\chi_c((L\cap H_n)\backslash\bigcup_{i=1}^{n-1}H_i)=
$$
$$
=\chi_c(L\backslash\bigcup_{i=1}^{n-1}H_i)-\chi_c((L\cap H_n)\backslash\bigcup_{i=1}^{n-1}H_i)=\chi_c(L\backslash\bigcup_{i=1}^{n}H_i)
$$
by induction hypothesis.

More precisely, we have
$$
\chi_c(L\otimes\bar k,\LL_{\chi_1}\boxtimes\cdots\boxtimes\LL_{\chi_n})=\chi_c(L\backslash H)=1-\chi_c(L\cap(\bigcup_{i=1}^n H_i))=
$$
$$
=1+\left(\sum_{j=1}^n (-1)^{j}\sum_{\stackrel{I\subseteq\{1,\ldots,n\}}{|I|=j}}\chi_c(L\cap H_I)\right).
$$

For every $I\subseteq\{1,\ldots,n\}$, $L\cap H_I$ is either empty (which is always the case for $|I|>d$ since $L$ is in general position among its translates) or it is an affine subvariety of $\AAA^n_k$, which has Euler characteristic $1$. Therefore we get
$$
\chi_c(L\otimes\bar k,\LL_{\chi_1}\boxtimes\cdots\boxtimes\LL_{\chi_n})= 1+\sum_{j=1}^{d}(-1)^j a_j
$$
where $a_j$ is the number of subsets $I\subseteq\{1,\ldots,n\}$ with $j$ elements such that $L\cap H_I\neq\emptyset$. In particular, if $L$ is in general position, $a_j={n\choose j}$ for every $j=1,\ldots,d$, so
$$
\chi_c(L\otimes\bar k,\LL_{\chi_1}\boxtimes\cdots\boxtimes\LL_{\chi_n})= 1+\sum_{j=1}^{d}(-1)^j {n\choose j}=
$$
$$
=1+\sum_{j=1}^{d}(-1)^j \left({n-1\choose j-1}+{n-1\choose j}\right)=(-1)^d{n-1\choose d}.
$$
This completes the proof of Theorem \ref{thm1} (and of the first statement of Theorem \ref{thm2}).

\section{Proof of Theorem \ref{thm2}}
In this section we will compute the absolute values of the reciprocal roots of the $L$ function $L(L;\chi_1,\ldots,\chi_n)$ in the case where $L$ is in general position, proving Theorem \ref{thm2}.

We proceed by induction on $n$, the case $n=1$ (or, in general, $d=0$ or $d=n$) being trivial. Assume $n>1$ and $0<d<n$. Suppose first that there is some $i\in\{1,\ldots,n\}$ such that the product of the $\chi_j$ for $j\neq i$ is non-trivial (this is always the case except when all $\chi_i$ are equal to the same character $\chi$ with $\chi^{n-1}={\mathbf 1}$). Without loss of generality, we may assume that $i=n$. Consider the projection $\pi=\pi_n:{L\subseteq\AAA^n_k\to\AAA^1_k}$. Since $L$ is in general position, $\pi(L)=\AAA^1_k$. By the projection formula, we have
$$
\R\pi_{L!}(\LL_{\chi_1}\boxtimes\cdots\boxtimes\LL_{\chi_n})\cong\LL_{\chi_n}\otimes\R\pi_{L!}(\LL_{\chi_1}\boxtimes\cdots\boxtimes\LL_{\chi_{n-1}}\boxtimes\Ql).
$$

Let $K:=\R\pi_{L!}(\LL_{\chi_1}\boxtimes\cdots\boxtimes\LL_{\chi_{n-1}}\boxtimes\Ql)\in\Dbc(\AAA^1_k,\Ql)$. Then $\R\Gamma_c(\AAA^1_k,K)=\R\Gamma_c(L,\LL_{\chi_1}\boxtimes\cdots\boxtimes\LL_{\chi_{n-1}}\boxtimes\Ql)=\R\Gamma_c(\varpi(L),\LL_{\chi_1}\boxtimes\cdots\boxtimes\LL_{\chi_{n-1}})$ where $\varpi:\AAA^n_k\to\AAA^{n-1}_k$ is the projection onto the first $n-1$ coordinates. The last equality holds because $\varpi:L\to\varpi(L)$ is an isomorphism (otherwise $L$ would be parallel to the coordinate axis $x_n$, contradicting the general position condition).

Note that $\varpi(L)\subseteq\AAA^{n-1}_k$ is also in general position: for every $I\subseteq\{1,\ldots,n-1\}$ with $|I|\leq d+1$, if we denote by $\tilde H_I$ the intersection of the coordinate hyperplanes $H_i\subseteq\AAA^{n-1}_k$ for $i\in I$, then $\dim(\varpi(L)\cap \tilde H_I)=\dim(L\cap\varpi^{-1}(\varpi(L)\cap\tilde H_I))=\dim(L\cap\varpi^{-1}(\tilde H_I))=\dim(L\cap H_I)=d-|I|$. Therefore, by theorem \ref{thm3}, $\HHH^i_c(\AAA^1_k,K)=\HHH^i_c(\varpi(L),\LL_{\chi_1}\boxtimes\cdots\boxtimes\LL_{\chi_{n-1}})$ vanishes for $i\neq d$, and $\HHH^d_c(\AAA^1_{\bar k},K)$ has dimension ${n-2\choose d}$ and is pure of weight $d$ by induction hypothesis.

Next, we check that for all but finitely many $t\in\bar k^\times$ the fibre $\pi^{-1}(t)\in\AAA^{n-1}_{\bar k}$ is in general position. Fix $I\subseteq\{1,\ldots,n-1\}$ with $|I|\leq d$. Then $\pi^{-1}(t)\cap \tilde H_I=L\cap H_I\cap \{x_n=t\}$. Since $L$ is in general position, $\dim(L\cap H_I)=d-|I|$. So, either $\dim(L\cap H_i\cap\{x_n=t\})=d-1-|I|$, or $L\cap H_I\subseteq \{x_n=t\}$. The latter is not possible if $|I|<d$, because then $L\cap H_{I\cup\{n\}}$ would be empty, contradicting the general position hypothesis. And, if $|I|=d$, it can only happen for at most ${n-1\choose d}$ values of $t$. Furthermore, the fibre at $t=0$ (and, more generally, the intersection of $L$ with any subset of the coordinate hyperplanes) is in general position (with respect to the remaining coordinate hyperplanes).

By induction hypothesis and Proposition \ref{smooth}, we conclude that there is a finite subset $Z\subseteq\AAA^1_{\bar k}$ not containing $0$ such that on $U:=\AAA^1_{\bar k}\backslash Z$ all cohomology sheaves of $K$ other than $\HH^{d-1}(K)$ vanish, and $\HH^{d-1}(K)$ is smooth of rank ${n-2\choose d-1}$ and pure of weight $d-1$. Furthermore, by the previous section $K$ is the restriction to a line of a perverse sheaf on $\AAA^{m}_k$ (namely, $RA'_!(\LL_{\chi_1}\boxtimes\cdots\boxtimes\LL_{\chi_{n-1}}))$ where $A'$ is the matrix $A$ with its last column deleted), which implies that $\HH^i(K)=0$ for $i<d-1$ and $\HH^{d-1}(K)$ has no punctual sections \cite[Propositions 7 and 9]{katz2003semicontinuity}. Let $\FF:=\HH^{d-1}(K)$. From the truncation distinguished triangle
$$
\FF[1-d]\to K\to\tau_{\geq d}K\to.
$$
we get
$$
\R\Gamma_c(\AAA^1_{\bar k},\FF)[1-d]\to\R\Gamma_c(\AAA^1_{\bar k},K)\to\R\Gamma_c(\AAA^1_{\bar k},\tau_{\geq d}K)\cong\bigoplus_{z\in Z}\tau_{\geq d}K_z
$$
since $\tau_{\geq d}K$ is supported on $Z$. This gives an exact sequence
$$
0\to\HHH^1_c(\AAA^1_{\bar k},\FF)\to\HHH^d_c(\AAA^1_{\bar k},K)\to\bigoplus_{z\in Z}\HH^d(K)_z\to\HHH^2_c(\AAA^1_{\bar k},\FF)\to 0
$$
and
$$
\bigoplus_{z\in Z}\HH^i(K)_z\cong\HHH^{i+2-d}_c(\AAA^1_{\bar k},\FF)=0
$$
for $i>d$. Also, $\HHH^2_c(\AAA^1_{\bar k},\FF)=\HHH^2_c(U,\FF)$ is pure of weight $d+1$ since $\FF$ is smooth and pure of weight $d-1$ on $U$, so the map $\bigoplus_{z\in Z}\HH^d(K)_z\to\HHH^2_c(\AAA^1_{\bar k},\FF)$ must be trivial as $\HH^d(K)_z$ has weights $\leq d$. Therefore $\HHH^2_c(\AAA^1_{\bar k},\FF)=0$ and the exact sequence reduces to
$$
0\to\HHH^1_c(\AAA^1_{\bar k},\FF)\to\HHH^d_c(\AAA^1_{\bar k},K)\to\bigoplus_{z\in Z}\HH^d(K)_z\to 0.
$$
Since $\HHH^d_c(\AAA^1_{\bar k},K)$ is pure of weight $d$, so are $\HHH^1_c(\AAA^1_{\bar k},\FF)$ and $\bigoplus_{z\in Z}\HH^d(K)_z$. If we denote by $j:U\hookrightarrow\PP^1_{\bar k}$ the inclusion, we have an exact sequence
$$
0\to\FF\to j_\ast j^\ast\FF \to (j_\ast j^\ast\FF)/\FF\to 0
$$
since $\FF$ has no punctual sections, which gives rise to an exact sequence (where we identify $\FF$ and its extension by zero to $\PP^1_{\bar k}$)
$$
0\to\HHH^0(Z\cup\{\infty\},(j_\ast j^\ast\FF)/\FF)\to\HHH^1_c(\AAA^1_{\bar k},\FF)\to\HHH^1(\PP^1_{\bar k},j_\ast j^\ast\FF)\to 0
$$
and, since $\HHH^0(Z\cup\{\infty\},(j_\ast j^\ast\FF)/\FF)$ has weights $\leq d-1$, we conclude that $(j_\ast j^\ast\FF)/\FF=0$ and $\FF\cong j_\ast j^\ast\FF$.

We now tensor the distinguished triangle
$$
\FF[1-d]\to K\to\tau_{\geq d}K=\bigoplus_{z\in Z}\HH^d(K)_z[-d]\to.
$$
with $\LL_{\chi_n}$ to obtain a triangle
$$
\LL_{\chi_n}\otimes\FF[1-d]\to \LL_{\chi_n}\otimes K\to\bigoplus_{z\in Z}\LL_{\chi_n}\otimes\HH^d(K)_z[-d]\to.
$$

We know that $\HHH^i_c(\AAA^1_{\bar k},\LL_{\chi_n}\otimes K)=\HHH^i_c(L\otimes\bar k,\LL_{\chi_1}\boxtimes\cdots\boxtimes\LL_{\chi_n})=0$ for $i\neq d$, and $\HHH^i_c(\AAA^1_{\bar k},\bigoplus_{z\in Z}\LL_{\chi_n}\otimes\HH^d(K)_z[-d])$ also vanishes for $i\neq d$ since $\bigoplus_{z\in Z}\LL_{\chi_n}\otimes\HH^d(K)_z$ is punctual. We have therefore an exact sequence
$$
0\to\HHH^1_c(\AAA^1_{\bar k},\LL_{\chi_n}\otimes\FF)\to\HHH^d_c(L\otimes\bar k,\LL_{\chi_1}\boxtimes\cdots\boxtimes\LL_{\chi_n})\to \bigoplus_{z\in Z}\LL_{\chi_n}\otimes\HH^d(K)_z\to
$$
$$
\to\HHH^2_c(\AAA^1_{\bar k},\LL_{\chi_n}\otimes\FF)\to 0
$$
and the same argument as above shows that $\HHH^2_c(\AAA^1_{\bar k},\LL_{\chi_n}\otimes\FF)=0$. Since $\bigoplus_{z\in Z}\LL_{\chi_n}\otimes\HH^d(K)_z=\LL_{\chi_n}\otimes(\bigoplus_{z\in Z}\HH^d(K)_z)$ is pure of weight $d$, we conclude that the weight $<d$ subspace of $\HHH^d_c(L\otimes\bar k,\LL_{\chi_1}\boxtimes\cdots\boxtimes\LL_{\chi_n})$ is isomorphic to that of $\HHH^1_c(\AAA^1_{\bar k},\LL_{\chi_n}\otimes\FF)$.

Let $j:\GG_{m,k}\hookrightarrow\PP^1_k$ be the inclusion. Since $\FF$ is isomorphic to $j_\ast j^\ast\FF$ and $\LL_{\chi_n}$ is smooth on $\GG_{m,k}$, $j_\ast j^\ast(\LL_{\chi_n}\otimes\FF)\cong\LL_{\chi_n}\otimes j_\ast j^\ast\FF\cong\LL_{\chi_n}\otimes\FF$ on $\GG_{m,k}$, and we get an exact sequence
$$
0\to \LL_{\chi_n}\otimes\FF\to j_\ast j^\ast(\LL_{\chi_n}\otimes\FF)\to i_{0\ast}i_{0}^\ast j_\ast(\LL_{\chi_n}\otimes\FF)\oplus i_{\infty\ast}i_{\infty}^\ast j_\ast(\LL_{\chi_n}\otimes\FF)\to 0
$$
and, taking cohomology,
$$
0\to j_\ast(\LL_{\chi_n}\otimes\FF)_0\oplus j_\ast(\LL_{\chi_n}\otimes\FF)_\infty\to\HHH^1_c(\AAA^1_{\bar k},\LL_{\chi_n}\otimes\FF)\to\HHH^1_c(\PP^1_{\bar k},j_\ast j^\ast(\LL_{\chi_n}\otimes\FF))\to 0.
$$

But $\LL_{\chi_n}\otimes\FF$ is totally ramified at $0$ (since $\FF$ is unramified and $\LL_{\chi_n}$ is totally ramified), so $j_\ast(\LL_{\chi_n}\otimes\FF)_0=0$. On the other hand, $\HHH^1_c(\PP^1_{\bar k},j_\ast j^\ast(\LL_{\chi_n}\otimes\FF))$ is pure of weight $d$ by \cite[Théorème 2]{deligne1980conjecture}. We conclude that the weight $<d$ subspace of $\HHH^1_c(\AAA^1_{\bar k},\LL_{\chi_n}\otimes\FF)$ is isomorphic to $j_\ast(\LL_{\chi_n}\otimes\FF)_\infty$.

We claim that, as a representation of the inertia group $I_\infty$ of $\PP^1_{\bar k}$ at infinity, $\LL_{\chi_n}\otimes\FF$ is isomorphic to a direct sum of ${n-2}\choose{d-1}$ copies of $\LL_{\chi_1\cdots\chi_{n}}\cong\LL_{\chi_1}\otimes\cdots\otimes\LL_{\chi_{n}}$. Equivalently, $\LL_{(\chi_1\cdots\chi_{n-1})^{-1}}\otimes\FF$ is unramified at infinity. We have
$$
\LL_{(\chi_1\cdots\chi_{n-1})^{-1}}\otimes\FF=\LL_{(\chi_1\cdots\chi_{n-1})^{-1}}\otimes\R^{d-1}\pi_{L!}(\LL_{\chi_1(x_1)}\otimes\cdots\otimes\LL_{\chi_{n-1}(x_{n-1})})\cong
$$
$$
\cong\R^{d-1}\pi_{L!}(\LL_{\chi_1(x_1/x_n)}\otimes\cdots\otimes\LL_{\chi_{n-1}(x_{n-1}/x_n)})\cong
$$
$$
\cong
\R^{d-1}\pi_{M!}(\LL_{\chi_1(x_1)}\otimes\cdots\otimes\LL_{\chi_{n-1}(x_{n-1})})
$$
where $M\subseteq\AAA^n_k$ is $\sigma^{-1}(L)$, $\sigma:\AAA^n_k\to\AAA^n_k$ is the map $(x_1,\ldots,x_{n-1},x_n)\mapsto(x_1x_n,\ldots,x_{n-1}x_n,x_n)$ (which is an isomorphism away form the coordinate hyperplanes) and $\pi_M$ is the restriction of $\pi$ to $M$. If $A\cdot{\mathbf x}={\mathbf b}$ is a system of independent linear equations that define $L$, $M$ is defined by $(\mathbf a_1 x_1+\cdots+\mathbf a_{n-1}x_{n-1}+\mathbf a_n)x_n=\mathbf b$, where $\mathbf a_1,\ldots,\mathbf a_n$ are the columns of the matrix $A$. If $\iota:\GG_{m,k}\to\GG_{m,k}$ denotes the inversion map, we have then 
$$
\iota^\ast(\LL_{(\chi_1\cdots\chi_{n-1})^{-1}}\otimes\FF)\cong\R^{d-1}\pi_{N!}(\LL_{\chi_1(x_1)}\otimes\cdots\otimes\LL_{\chi_{n-1}(x_{n-1})})
$$
where $N\subseteq\AAA^n_k$ is defined by $\mathbf a_1 x_1+\cdots+\mathbf a_{n-1}x_{n-1}+\mathbf a_n=\mathbf b x_n$ or $A'\cdot\mathbf x=\mathbf b x_n-\mathbf a_n$, where $A'$ is the matrix $A$ with the last column deleted. That is, $\iota^\ast(\LL_{(\chi_1\cdots\chi_{n-1})^{-1}}\otimes\FF)$ is isomorphic to the restriction via the map $\GG_{m,k}\to\AAA^{n-d}_k$ given by $\lambda\mapsto \lambda\mathbf b-\mathbf a_n$ of the sheaf $\R^{d-1}A'_!(\LL_{\chi_1}\boxtimes\cdots\boxtimes\LL_{\chi_{n-1}})$ on $\AAA^{n-d}_k$. But this sheaf is smooth at $-\mathbf a_n$ as seen in the previous section, since the affine variety defined by $A'\cdot\mathbf x=-\mathbf a_n$ is in general position (all size $(n-r)$ minors of the augmented matrix are minors of the matrix $A$, and therefore non-zero since $L$ itself is in general position). So $\iota^\ast(\LL_{(\chi_1\cdots\chi_{n-1})^{-1}}\otimes\FF)$ extends to a sheaf on $\AAA^1_k$ which is smooth at $0$ (and pure of weight $d-1$ by induction hypothesis), so it is unramified there. Hence $\LL_{(\chi_1\cdots\chi_{n-1})^{-1}}\otimes\FF$ is unramified at infinity.

Thus, if $\chi_1\cdots\chi_n$ is non-trivial, $\LL_{\chi_n}\otimes\FF\cong\LL_{\chi_1\cdots\chi_n}\otimes(\LL_{(\chi_1\cdots\chi_{n-1})^{-1}}\otimes\FF)$ is totally ramified at infinity, so $j_\ast(\LL_{\chi_n}\otimes\FF)_\infty=0$ and $\HHH^1_c(\AAA^1_{\bar k},\LL_{\chi_n}\otimes\FF)$ is pure of weight $d$. On the other hand, if $\chi_1\cdots\chi_n$ is trivial, $\LL_{\chi_n}\otimes\FF$ is unramified at infinity, so $j_\ast(\LL_{\chi_n}\otimes\FF)_\infty$ has dimension $\mathrm{rank}(\FF)={{n-2}\choose{d-1}}$ and, as seen above, is pure of weight $d-1$. We conclude that $\HHH^1_c(\AAA^1_{\bar k},\LL_{\chi_n}\otimes\FF)$ has ${{n-2}\choose{d-1}}$ Frobenius eigenvalues of weight $d-1$ and ${{n-1}\choose d}-{{n-2}\choose{d-1}}={{n-2}\choose{d}}$ of weight $d$.

It remains to prove the case where all $\chi_i$ are equal to a fixed non-trivial $\chi$ such that $\chi^{n-1}=\mathbf 1$, in which case we need to show that $\HHH^d_c(\AAA^1_{\bar k},K\otimes\LL_{\chi})$ is pure of weight $d$. By the previous case, we know that $\HHH^d_c(\AAA^1_{\bar k},K\otimes\LL_{\eta})\cong\HHH^d_c(L\otimes\bar k,\LL_{\chi_1}\boxtimes\cdots\boxtimes\LL_{\chi_{n-1}}\boxtimes\LL_\eta)$ is pure of weight $d$ for every non-trivial character $\eta$ of $k^\times$ other than $\chi$.

Like in the previous case, the distinguished triangle
$$
\FF[1-d]\to K\to\tau_{\geq d}K\to.
$$
gives rise to an exact sequence
$$
0\to\HHH^1_c(\AAA^1_{\bar k},\FF)\to\HHH^d_c(\AAA^1_{\bar k},K)\to\bigoplus_{z\in Z}\HH^d(K)_z\to\HHH^2_c(\AAA^1_{\bar k},\FF)\to 0
$$
and, similarly,
$$
0\to\HHH^1_c(\AAA^1_{\bar k},\FF\otimes\LL_\eta)\to\HHH^d_c(\AAA^1_{\bar k},K\otimes\LL_\eta)\to\bigoplus_{z\in Z}\HH^d(K\otimes\LL_\eta)_z\to 0
$$
for any non-trivial character $\eta$, since $\HHH^2_c(\AAA^1_{\bar k},\FF\otimes\LL_\eta)=0$ because $\FF$ is unramified at $0$. In particular, $\HHH^1_c(\AAA^1_{\bar k},\FF\otimes\LL_\eta)$ and $\bigoplus_{z\in Z}\HH^d(K\otimes\LL_\eta)_z$ are pure of weight $d$ for $\eta\neq\chi$, and therefore so is $\bigoplus_{z\in Z}\HH^d(K\otimes\LL_\chi)_z=\LL_{\chi/\eta}\otimes\left(\bigoplus_{z\in Z}\HH^d(K\otimes\LL_\eta)_z\right)$. So $\HHH^d_c(\AAA^1_{\bar k},K\otimes\LL_\chi)$ is pure of weight $d$ if and only if $\HHH^1_c(\AAA^1_{\bar k},\FF\otimes\LL_\chi)$ is. 

Next, we consider the exact sequence
$$
0\mapsto j_!j^\ast\FF\to\FF\to i_\ast i^\ast\FF\to 0
$$
where $j:U\hookrightarrow\AAA^1_{\bar k}$ and $i:Z\hookrightarrow\AAA^1_{\bar k}$ are the inclusions, which gives rise to a cohomology exact sequence
$$
0\to\bigoplus_{z\in Z}(\FF\otimes\LL_\eta)_z\to\HHH^1_c(U,\FF\otimes\LL_\eta)\to\HHH^1_c(\AAA^1_{\bar k},\FF\otimes\LL_\eta)\to 0
$$
for any non-trivial character $\eta$, since $\FF$ has no punctual sections. Since the weights of $\bigoplus_{z\in Z}(\FF\otimes\LL_\eta)_z$ clearly do not depend on $\eta$ (because $\LL_\eta$ is always pure of weight $0$), the result is then a consequence of the following lemma, which we can apply in this case given that $\FF\cong\LL_{(\chi_1\cdots\chi_{n-1})^{-1}}\otimes\FF$ is unramified at infinity as shown in the proof of the previous case.

\begin{lemma}
 Let $\FF$ be a smooth $\Ql$-sheaf on some non-empty open set $U\subseteq\AAA^1_k$, mixed of some weights and unramified at $0$ and $\infty$. Then the dimension and the weights of $\HHH^1_c(U,\FF\otimes\LL_\eta)$, for $\eta$ a non-trivial multiplicative character of $k^\times$, are independent of $\eta$.
\end{lemma}

\begin{proof}
 Since $\LL_\eta$ vanishes at $0$, we can and will assume that $0\not\in U$. Since $\FF$ is smooth an unramified at $0$ and $\infty$, $\FF\otimes\LL_\eta$ is totally ramified at both points, so $\HHH^2_c(U,\FF\otimes\LL_\eta)=0$ and the Euler characteristic of $\FF\otimes\LL_\eta$ is $-\dim\HHH^1_c(U,\FF\otimes\LL_\eta)$. By the Grothendieck-Ogg-Shafarevic formula, this Euler characteristic is $(\# Z)\cdot\mathrm{rk}(\FF\otimes\LL_\eta)+\sum_{z\in Z\cup\{0,\infty\}}\mathrm{swan}_z(\FF\otimes\LL_\eta)$ where $Z:=\GG_{m
 ,\bar k}\backslash U$, which is clearly independent of $\eta$ (tensoring with a rank one tame representation does not change the Swan conductor).
 
 For any exact sequence $0\to\GGG\to\FF\to\HH\to 0$ of smooth sheaves we get a cohomology exact sequence $0\to\HHH^1_c(U,\GGG\otimes\LL_\eta)\to\HHH^1_c(U,\FF\otimes\LL_\eta)\to\HHH^1_c(U,\HH\otimes\LL_\eta)\to 0$ (since both $\GGG\otimes\LL_\eta$ and $\HH\otimes\LL_\eta$ are also totally ramified at $0$ and $\infty$) so we may assume that $\FF$ is pure of some weight $w$.
 
 Let $j:U\hookrightarrow\PP^1_k$ be the inclusion. Note that, since $\FF\otimes\LL_\eta$ is totally ramified at $0$ and $\infty$, $j_\ast\FF$ vanishes at both points. If we denote by $i:Z\hookrightarrow\PP^1_k$ the inclusion, the exact sequence of sheaves
 $$
 0\to j_!(\FF\otimes\LL_\eta)\to j_\ast(\FF\otimes\LL_\eta)\to i_\ast i^\ast j_\ast(\FF\otimes\LL_\eta)\to 0
 $$
 induces an exact sequence
 $$
 0\to\HHH^0_c(Z,i^\ast j_\ast(\FF\otimes\LL_\eta))\to\HHH^1_c(U,\FF\otimes\LL_\eta)\to\HHH^1(\PP^1_{\bar k},j_\ast(\FF\otimes\LL_\eta))\to 0
 $$
 where the last non-zero term is pure of weight $w+1$ by \cite[Théorème 2]{deligne1980conjecture}. So it remains to show that the weights of $\HHH^0_c(Z,i^\ast j_\ast(\FF\otimes\LL_\eta))=\oplus_{z\in Z} j_\ast(\FF\otimes\LL_\eta)_z$ are independent of $\eta$. But $\LL_\eta$ is smooth at every $z\in Z$, so $j_\ast(\FF\otimes\LL_\eta)_z\cong \LL_{\eta,z}\otimes (j_\ast\FF)_z$ has the same weights as $(j_\ast\FF)_z$, which are independent of $\eta$.
\end{proof}

\bibliographystyle{amsalpha}
\bibliography{bibliography}

\end{document}